\documentclass[final,12pt]{colt2018}

\usepackage{times}

\def\colorful{0}

\usepackage{amsfonts}
\usepackage{color}
\usepackage{float}
\usepackage{verbatim}
\usepackage{enumitem}
\usepackage{multirow}

\newif\ifhyper\IfFileExists{hyperref.sty}{\hypertrue}{\hyperfalse}
\hypertrue
\ifhyper\usepackage{hyperref}\fi

\usepackage{enumitem}

\usepackage{framed}
\usepackage{nicefrac}

\def\nnewcolor{0}
\ifnum\nnewcolor=1
\newcommand{\nnew}[1]{{\color{red} #1}}
\fi
\ifnum\nnewcolor=0
\newcommand{\nnew}[1]{#1}
\fi

\ifnum\colorful=1
\newcommand{\new}[1]{{\color{red} #1}}

\else
\newcommand{\new}[1]{{#1}}

\fi

\newtheorem{informal theorem}[theorem]{Theorem (informal statement)}

\newtheorem{claim}[theorem]{Claim}

\newcommand{\eqdef}{\stackrel{{\mathrm {\footnotesize def}}}{=}}

\newcommand{\R}{\mathbb{R}}

\newcommand{\Z}{\mathbb{Z}}

\newcommand{\E}{\mathbf{E}}
\newcommand{\eps}{\epsilon}
\newcommand{\dtv}{d_{\mathrm TV}}

\newcommand{\TV}{\mathrm{d_{TV}}}
\newcommand{\KL}{\mathrm{KL}}

\newcommand{\pini}{P^{\mathrm{in}}_i}

\newcommand{\pouti}{P^{\mathrm{out}}_i}
\newcommand{\cin}{C^\mathrm{in}}
\newcommand{\cout}{C^\mathrm{out}}

\newcommand{\mle}{\mathop{\hat{f}_n}}
\newcommand{\pmin}{\mathop{p_{\mathrm{min}}}}

\newcommand{\vol}{\mathrm{vol}}

\title[On the MLE of Multivariate Log-concave Densities]{Near-Optimal Sample Complexity Bounds for Maximum Likelihood Estimation of 
Multivariate Log-concave Densities}

\coltauthor{\Name{Timothy Carpenter} \Email{carpenter.454@osu.edu}\\
 \addr The Ohio State University
 \AND
 \Name{Ilias Diakonikolas} \Email{diakonik@usc.edu}\\
 \addr University of Southern California
  \AND
 \Name{Anastasios Sidiropoulos} \Email{sidiropo@gmail.com}\\
 \addr University of Illinois at Chicago
  \AND
 \Name{Alistair Stewart} \Email{alistais@usc.edu}\\
 \addr University of Southern California
 }

\begin{document}

\maketitle




\begin{abstract}
We study the problem of learning multivariate log-concave densities
with respect to a global loss function. We obtain the first upper bound on the sample complexity 
of the maximum likelihood estimator (MLE) for a log-concave density on $\R^d$, for all $d \geq 4$.
Prior to this work, no finite sample upper bound was known for this estimator in more than $3$ dimensions.

In more detail, we prove that for any $d \geq 4$ and $\eps>0$, given 
$\tilde{O}_d((1/\eps)^{(d+3)/2})$ samples drawn from an unknown log-concave density $f_0$ on $\R^d$,
the MLE outputs a hypothesis $h$ that with high probability is $\eps$-close
to $f_0$, in squared Hellinger loss. For any $d \geq 2$, a sample complexity lower bound of $\Omega_d((1/\eps)^{(d+1)/2})$
was previously known for any learning algorithm that achieves this guarantee. 
We thus establish that the sample complexity of the log-concave MLE is near-optimal for $d \geq 4$, 
up to an $\tilde{O}(1/\eps)$ factor.
\end{abstract}



\section{Introduction} \label{sec:intro}

\subsection{Background}

The general task of estimating a probability distribution under certain 
qualitative assumptions about the {\em shape} of its probability density function
has a long history in statistics, dating back to the pioneering work of 
~\cite{Grenander:56} 
who analyzed the maximum likelihood estimator of a univariate monotone density. 
Since then, shape constrained density estimation has been a very active research area 
with a rich literature in mathematical statistics and, more recently, in computer science. 
A wide range of shape constraints have been studied, including unimodality, convexity and concavity, 
$k$-modality, log-concavity, and $k$-monotonicity.
The reader is referred to~\cite{BBBB:72} for a summary of the early work and to~\cite{GJ:14} 
for a recent book on the subject. (See Section~\ref{ssec:related} for a succinct summary of prior work.)
The majority of the literature has studied the univariate (one-dimensional) setting, 
which is by now fairly well-understood for a range of distributions. On the other hand,
the multivariate setting and specifically the regime of {\em fixed} dimension is significantly more
challenging and poorly understood for many natural distribution families.

In this work, we focus on the family of {\em multivariate} log-concave distributions. A distribution on $\R^d$
is log-concave if the logarithm of its probability density function is concave (see Definition~\ref{def:lc}).
Log-concave distributions constitute a rich non-parametric family
encompassing a range of fundamental distributions, including
uniform, normal, exponential, logistic, extreme value,
Laplace, Weibull, Gamma, Chi and Chi-Squared, and Beta distributions (see, e.g.,~\cite{BagnoliBergstrom05}).
Due to their fundamental nature and appealing properties,
log-concave distributions have been studied in a range of fields including
economics~\cite{An:95}, probability theory~\cite{SW14-survey},
computer science~\cite{LV07}, and geometry~\cite{Stanley:89}.

The problem of {\em density estimation} for log-concave distributions is of central importance
in the area of non-parametric shape constrained estimation~\cite{Walther09, SW14-survey, Sam17-survey}
and has received significant attention during the past decade in statistics~\cite{Cule10a, 
DumbgenRufibach:09, DossW16, ChenSam13, KimSam16, BalDoss14, HW16} 
and theoretical computer science~\cite{CDSS13, CDSS14, ADLS17, CanonneDGR16, DKS16-proper-lc, DiakonikolasKS17-lc}.

\subsection{Our Results and Comparison to Prior Work}
In this work, we analyze the global convergence rate of the maximum likelihood estimator (MLE)
of a multivariate log-concave density. Formally, we study the following fundamental question: 
\begin{center}
{\em How many samples are information-theoretically sufficient
so that the MLE of an arbitrary \\ log-concave density on $\R^d$ learns the underlying density, within squared Hellinger loss $\eps$?}
\end{center}
Perhaps surprisingly, despite significant effort within the statistics community 
on analyzing the log-concave MLE, our understanding of its finite sample performance 
in constant dimension has remained poor. 
The only result prior to this work that addressed the sample complexity 
of the MLE in more than one dimensions is by 
~\cite{KimSam16}.
Specifically,~\cite{KimSam16} obtained the following results: 
\begin{itemize}
\item[(1)] a sample complexity {\em lower bound} of
$\Omega_d \left( (1/\eps)^{(d+1)/2} \right)$ that applies to {\em any} estimator 
for all $d \geq 2$, and 
\item[(2)] a sample complexity {\em upper bound} for the log-concave MLE, that is near-optimal (within logarithmic factors) 
for $d \leq 3$. 
\end{itemize}
{\em Prior to our work, no finite sample upper bound was known for the log-concave 
MLE even for $d=4$.}

In recent related work, 
~\cite{DiakonikolasKS17-lc} 
established a finite sample complexity upper bound for learning multivariate 
log-concave densities under global loss functions. 
Specifically, the estimator analyzed in~\cite{DiakonikolasKS17-lc} 
uses $\tilde{O}_d \left( (1/\eps)^{(d+5)/2} \right)$\footnote{The $\tilde{O}(\cdot)$ notation hides 
logarithmic factors in its argument.} samples
and learns a log-concave density on $\R^d$ within squared Hellinger loss 
$\eps$, with high probability. 
We remark that the upper bound of~\cite{DiakonikolasKS17-lc}
was obtained by analyzing an estimator that is {\em substantially different} than the log-concave MLE. 
Moreover, the analysis in~\cite{DiakonikolasKS17-lc} has no implications on the performance of the MLE. 
Interestingly, some of the technical tools employed in~\cite{DiakonikolasKS17-lc} 
will be useful in our current setting.

Due to the fundamental nature of the MLE, understanding its performance merits investigation
in its own right. In particular, the log-concave MLE has an intriguing geometric structure
that is a topic of current investigation~\cite{Cule10a, RSU17}.
The output of the log-concave MLE 
satisfies several desirable properties 
that may not be automatically satisfied by surrogate estimators. 
These include the log-concavity of the hypothesis, 
the paradigm of log-concave projections and their continuity in Wasserstein distance, affine equivariance, 
one-dimensional characterization, and adaptation (see, e.g.,~\cite{Sam17-survey}). 
An additional motivation comes from a recent conjecture (see, e.g.,~\cite{Wellner15}) that 
for $4$-dimensional log-concave densities the MLE may have sub-optimal sample complexity.  
These facts provide strong motivation for characterizing the sample complexity of the log-concave MLE
in any dimension.

To formally state our results, we will need some terminology.
The {\em squared Hellinger distance} between two density functions $f, g: \R^d \to \R_+$ is defined as 
$h^2(f, g) = (1/2) \cdot \int_{\R^d} ( \sqrt{f(x)} - \sqrt{g(x)})^2 dx$.

We now define our two main objects of study:

\begin{definition}[Log-concave Density] \label{def:lc}
A probability density function $f : \R^d \to \R_+$, $d \in \Z_+$, is called {\em log-concave}
if there exists an upper semi-continuous concave function $\phi: \R^d \to [-\infty, \infty)$
such that $f(x) = e^{\phi(x)}$ for all $x \in \R^d$.
We will denote by $\mathcal{F}_d$ the set of upper semi-continuous,
log-concave densities with respect to the Lebesgue
measure on $\R^d$.
\end{definition}

\begin{definition}[Log-concave MLE] \label{def:mle}
Let $f_0 \in \mathcal{F}_d$ and $X_1, \ldots, X_n$ be iid samples from $f_0$.
The {\em maximum likelihood estimator}, $\mle$, 
is the density $\mle \in \mathcal{F}_d$ which maximizes
$\frac{1}{n} \sum_{i=1}^n \log(f(X_i))$ over all $f \in \mathcal{F}_d$.
\end{definition}

\noindent We can now state our main result:

\begin{theorem}[Main Result] \label{thm:main}
Fix $d \in \Z_+$ and $\eps \in (0,1)$.
Let $ n = \nnew{\tilde{\Omega}}_d \left( (1/\eps)^{(d+3)/2} \right)$.
For any $f_0 \in  \mathcal{F}_d$, with probability at least $9/10$ over the $n$ samples from $f_0$, we have that
$h^2(\mle, f_0) \leq \eps$.
\end{theorem}

\noindent See Theorem~\ref{thm:h-main-result} for a more detailed statement.
The aforementioned lower bound of~\cite{KimSam16} implies
that our upper bound is tight up to an $\tilde{O}_d (\eps^{-1})$ multiplicative factor.


\subsection{Related Work} \label{ssec:related}
Shape constrained density estimation is a vibrant research field within mathematical statistics. 
Statistical research in this area started in the 1950s and has seen a recent surge of
research activity, in part due to the ubiquity of structured distributions in various domains. 
The standard method used in statistics to address density estimation problems of this form 
is the MLE.
See~\cite{Brunk:58, PrakasaRao:69, Wegman:70, 
HansonP:76, Groeneboom:85, Birge:87, Birge:87b,Fougeres:97,ChanTong:04,BW07aos, JW:09, 
DumbgenRufibach:09, BRW:09aos, GW09sc, BW10sn, KoenkerM:10aos, Walther09, ChenSam13, KimSam16, BalDoss14, HW16}
for a partial list of works analyzing the MLE for various distribution families.
During the past decade, there has been a large body of work on shape constrained density estimation 
in computer science with a focus on both sample and computational efficiency~\cite{DDS12soda, DDS12stoc, DDOST13focs, CDSS13, CDSS14, CDSS14b, 
ADHLS15, ADLS17, DKS15, DKS15b, DDKT15, DKS16, VV16, DiakonikolasKS17-lc}.

Density estimation of log-concave densities has been extensively investigated.
The univariate case is by now well understood~\cite{DL:01, CDSS14, ADLS17, KimSam16, HW16}. 
For example, it is known~\cite{KimSam16, HW16} that $\Theta(\eps^{-5/4})$ samples
are necessary and sufficient to learn an arbitrary log-concave density over $\R$ within squared Hellinger loss $\eps$.
Moreover, the MLE is sample-efficient~\cite{KimSam16, HW16} 
and attains certain adaptivity properties~\cite{KGS16}. A recent line of work in computer 
science~\cite{CDSS13, CDSS14, ADLS17, CanonneDGR16, DKS16-proper-lc}
gave efficient algorithms for log-concave density estimation 
under the total variation distance. 

Density estimation of multivariate log-concave densities has been systematically studied as well.
A line of work~\cite{Cule10a, DumbgenRufibach:09, DossW16, ChenSam13, BalDoss14} has obtained
a complete understanding of the global consistency properties of the MLE for any dimension.
However, both the rate of convergence of the MLE and the minimax rate of convergence remain unknown for $d \geq 4$.
For $d \leq 3$,~\cite{KimSam16} show that the MLE is sample near-optimal (within logarithmic factors) 
under the squared Hellinger distance. \cite{KimSam16} also prove bracketing entropy lower bounds 
suggesting that the MLE may be sub-optimal for $d > 3$ (also see~\cite{Wellner15}).

\subsection{Technical Overview}
Here we provide a brief overview of our proof in tandem with a comparison to prior work.
\new{We start by noting that the previously known sample complexity upper bound of the log-concave 
MLE for $d \leq 3$~\cite{KimSam16}  was obtained by 
bounding from above the bracketing entropy of the class. As we explain below, our argument
is more direct making essential use of the VC inequality (Theorem~\ref{thm:vc}), 
a classical result from empirical process theory.} 
In contrast to prior work on log-concave density estimation
~\cite{KimSam16, DiakonikolasKS17-lc} which relied on approximations to (log)-concave
{\em functions}, we start by considering approximations to {\em convex sets}. 
Let $f_0$ be the target log-concave density. 
We show (Lemma~\ref{lem:alt_convex_set_prob}) that given sufficiently many samples from $f_0$, 
with high probability, for any convex set $C$ the empirical mass of $C$ 
and the probability mass of $C$ under $f_0$ are close to each other. 
We then leverage this structural lemma to analyze the error in the log-likelihood 
of log-concave densities, using the fact that the superlevel 
sets of a log-concave density are convex.

We remark that our aforementioned structural result (Lemma~\ref{lem:alt_convex_set_prob})
crucially requires the assumption of the log-concavity of $f_0$. Naively, one may think that 
this lemma follows directly from the VC inequality. 
Recall however that the VC-dimension of the family of convex sets is infinite, even in the plane.
For example, for the uniform distribution over the unit circle, 
a similar result does {\em not} hold for {\em any} finite number of samples (the intersection of the convex hull of any subset $S$ of the unit circle with the unit circle is $S$ itself, so we would need uniform convergence on all subsets of the unit circle), 
and so we need to use the fact that $f_0$ is log-concave.
To prove our lemma, we consider judicious approximations of 
the convex set $C$ with convex polytopes using known results from convex geometry. 
In more detail, we consider approximations to the convex set $C$ on the inside and outside 
with close probabilities under $f_0$ to the convex set from a family with a bounded VC-dimension.

For any log-concave density $f$, the probabilities of any superlevel 
set are close under the empirical distribution and $f_0$. If 
$\log f$ were bounded, then that would mean that the empirical log-likelihood of $f$
and the log-likelihood of $f$ under $f_0$ were close. Unfortunately, for any density $f$, 
$\log f$ is unbounded from below. 
To deal with this issue, we instead consider $\log(\max(f, p_{\min}))$, 
for some carefully chosen probability value $p_{\min}$ such that 
we could ignore the contribution of the density below 
$p_{\min}$ if $f$ is close to $f_0$. If we can bound the range of $\log(\max(f, p_{\min}))$, 
we can show that its expectation under $f_0$ and its empirical version 
are close to each other (see Lemma~\ref{lem:technical_f}).
To bound the range, we show that if the maximum value of $f$ is much larger than the maximum of $f_0$, 
then $f$ has small probability mass outside a set $A$ of small volume; since $A$ has small volume, 
we see many samples outside it, and so the empirical log-likelihood of $f$ is smaller than 
the empirical log-likelihood of $f_0$. Using this fact, we can show that for the MLE 
$\hat{f_n}$ the expectation of $\log(\max( \hat{f_n}, p_{\min}))$ 
is large under $f_0$ and then that $\hat{f_n}$ is close in Hellinger distance to $f_0$.

\subsection{Organization}
After setting up the required preliminaries in Section~\ref{sec:prelims}, in 
Section~\ref{sec:main-section} we present the proof of our main result, modulo 
the proof of our main lemma (Lemma~\ref{lem:alt_convex_set_prob}).
In Section~\ref{sec:warmup}, we give a slightly weaker version of 
Lemma~\ref{lem:alt_convex_set_prob} that has a significantly simpler proof.
In Section~\ref{sec:full-proof}, we present the proof of Lemma~\ref{lem:alt_convex_set_prob}.
Finally, we conclude with a few open problems in Section~\ref{sec:conc}.

\section{Preliminaries} \label{sec:prelims}

\noindent {\bf Notation and Definitions.}
\new{For $m \in \Z_+$, we denote $[m] \eqdef \{1, \ldots, m\}$.}
Let $f: \R^d \to \R$ be a Lebesgue measurable function.
We will use $f(A)$ to denote $\int_{A} f(x) dx$.
A Lebesgue measurable function $f: \R^d \to \R$ is a probability density function (pdf)
if $f(x) \geq 0$ for all $x \in \R^d$ and  $\int_{\R^d} f(x) dx = 1$.
Let $f, g: \R^d \to \R_+$ be probability density functions.
The {\em squared Hellinger distance} between $f, g$ is defined as $H^2(f, g) = \frac{1}{2} \int \left( \sqrt{f(x)} - \sqrt{g(x)} \right)^2 dx$.
The {\em total variation distance} between $f, g$ 
is defined as $\dtv(f, g) = \sup_{S} |f(S) - g(S)|$, where
the supremum is over all Lebesgue measurable subsets of the domain.
We have that $\dtv\left(f, g \right) = (1/2) \cdot \| f -g  \|_1 = (1/2) \cdot \int_{\R^d} |f(x) - g(x)| dx.$
The  {\em Kullback-Leibler (KL) divergence from $g$ to $f$} is defined as
$\KL(f || g) = \int_{-\infty}^{\infty} f(x) \ln \frac{f(x)}{g(x)} dx$.

\new{For $f : A \rightarrow B$ and $A' \subseteq A$, 
the restriction of $f$ to $A'$ is the function $f\vert_{A'} : A' \rightarrow B$.}
For $y \in [0,\infty)$ and $f: \R^d \to [0,\infty)$
we denote by $L_f(y) \eqdef \{x \in \R^d \mid f(x) \geq y\}$
its {\em superlevel sets}. If $f$ is log-concave, $L_f(y)$ is a convex set for all $y \in \R_+$.
For a function $f: \R^d \to [0,\infty)$, we will denote by $M_f$ its maximum value.

\smallskip

\noindent {\bf The VC inequality.}
We start by recalling the notion of VC dimension.
We say that a set $X \subseteq \mathbb{R}^d$ is \emph{shattered} by a collection $\mathcal{A}$
of subsets of $\R^d$, if for every $Y \subseteq X$ there exists $A \in \mathcal{A}$ such that $A \cap X = Y$.
The \emph{VC dimension} of a family $\mathcal{A}$ of subsets of $\mathbb{R}^d$ 
is defined to be the maximum cardinality of a subset $X \subseteq \mathbb{R}^d$ 
that is shattered by $\mathcal{A}$. If there is a shattered subset of size $s$ 
for all $s \in \mathbb{Z}_+$, then we say that the VC dimension of $\mathcal{A}$ is $\infty$.

The empirical distribution, $f_n$, corresponding to a density $f : \R^d \to \R_+$ 
is the discrete probability measure defined by 
$f_n(A)=(1/n) \cdot \sum_{i=1}^{n} \mathbf{1}_{A}(X_i)$, where the $X_i$ are iid samples drawn
from $f$ and  $\mathbf{1}_{S}$ is the characteristic function of the set $S$.
Let $f : \R^d \to \R$ be a Lebesgue measurable function.
Given a family $\mathcal{A}$ of measurable subsets of $\R^d$, 
we define the $\mathcal{A}$-norm of $f$ by $\|f\|_\mathcal{A} = \sup_{A \in \mathcal{A}}|f(A)|$.
The VC inequality states the following:
\begin{theorem}[VC inequality, see~\cite{DL:01}, p.~31]\label{thm:vc}
Let $f : \R^d \to [0,\infty)$ be a probability density function 
and $f_n$ be the empirical distribution obtained after drawing $n$ samples from $f$. 
Let $\mathcal{A}$ be a family of subsets over $\mathbb{R}^d$ with VC dimension $V$. 
Then $\E[\| f - f_n \|_{\mathcal{A}}] \leq C \sqrt{V / n}$, for some universal constant $C>0$.
\end{theorem}

We will also require a high probability version of the VC inequality
which can be obtained using the following standard uniform convergence bound:

\begin{theorem}[see~\cite{DL:01}, p.~17] \label{thm:A-expect-bound}
Let $\mathcal{A}$ be a family of subsets over $\R^d$ 
and $f_n$ be the empirical distribution of $n$ samples from the density $f: \R^d \to [0,\infty)$. 
Let $X$ be the random variable $\|f - f_n\|_{\mathcal{A}}$. 
Then for all $\delta>0$, we have that
$\Pr[X - \E[X] > \delta] \leq e^{-2n \delta^2}$.
\end{theorem}

\noindent {\bf Approximating Convex Sets by Polytopes.}
We make use of the following quantitative bounds of~\cite{GMR95} 
that provide volume approximation for any convex body 
by an inscribed and a circumscribed convex polytope respectively 
with a bounded number of facets:

\begin{theorem}\label{thm:inside_poly}\label{thm:outside_poly}
For any convex body $K \subseteq \mathbb{R}^d$, and $n$ sufficiently large, 
there exists a convex polytope $P \subseteq K$ with at most $\ell$ facets such that $\vol(K \setminus P) \leq \frac{\kappa d}{\ell^{2/(d-1)}}\vol(K)$, 
where $\kappa > 0$ is a universal constant.
Similarly, there exists a convex polytope $P'$ where $K \subseteq P'$ 
with at most $\ell$ facets such that $\vol(P' \setminus K) \leq \frac{\kappa d}{\ell^{2/(d-1)}}\vol(K)$.
\end{theorem}

\section{Main Result: Proof of Theorem~\ref{thm:main}} \label{sec:main-section}

The following theorem is a more detailed version of Theorem~\ref{thm:main} and is the main result of this paper:

\begin{theorem}\label{thm:h-main-result}
Fix $d \in \Z_+$ and \nnew{$\eps, \tau \in (0,1)$}. 
Let $n = \nnew{\Omega} \left( (d^2/\eps) \ln^3(d/(\eps\tau))  \right)^{(d+3)/2}$. 
For any $f_0 \in  \mathcal{F}_d$, with probability at least $1-\tau$ over the $n$ samples from $f_0$, we have that $h^2(\mle, f_0) \leq \eps$.
\end{theorem}

This section is devoted to the proof of Theorem~\ref{thm:h-main-result}, which follows from Lemma \ref{lem:mle_dist_N_1}.
We will require a sequence of intermediate lemmas and claims. 

\new{We summarize the notation that will appear throughout this proof. 
We use $f_0 \in \mathcal{F}_d$ to denote the target log-concave density.
We denote by $f_n$ the empirical distribution obtained after drawing $n$ iid samples $X_1, \ldots, X_n$ from $f_0$
and by $\mle$ the corresponding MLE. Given $d \in \Z_+$ and $0<\eps, \tau<1$, for concreteness, 
we will denote:
$$N_1 \eqdef \Theta\left( (d^2/\eps) \ln^3(d/(\eps\tau)) \right)^{(d+3)/2} \;,$$ 
for a sufficiently large universal constant in the big-$\Theta$ notation.
We will establish that $N_1$ is an upper bound on the desired sample complexity of the MLE.
Moreover, we will denote
$$z \eqdef \ln (100 n^4 /\tau^2) \;, \delta \eqdef \eps / (32 z) \;,$$ 
$$\pmin \eqdef M_{f_0} e^{-z} \;,$$
and 
$$S \eqdef L_{f_0}(\pmin) \;.$$
}

We start by establishing an upper bound on the volume of superlevel sets:

\begin{lemma}[see, e.g.,~\cite{DiakonikolasKS17-lc}, p.~8]\label{lem:DKS1}
Let $f \in \mathcal{F}_d$ with maximum value $M_f$. Then for all $w \geq 1$, we have
$
\vol(L_{f}(M_{f} e^{-w})) \leq \nnew{{w}^d / M_{f}},
$
and
$
\Pr_{X\sim f}[f(X) \leq M_{f} e^{-w}] \leq O(d)^d e^{-{w}/2}.
$
\end{lemma}

We defer this proof to Appendix \ref{sec:main-appendix}.
We use Lemma~\ref{lem:DKS1} to get a bound on the volume of the superlevel set that contains all the samples with high probability:

\begin{corollary}\label{lem:S_def}
For $n \geq N_1$, we have that:
\begin{itemize}
\item[(a)] $\vol(S) \leq \nnew{z^d / M_{f_0}}$, and 
\item[(b)] $\Pr_{X \sim f_0}[f_0(X) \leq M_{f_0}/(100 n^4  /\tau^2)] \leq \tau/\nnew{(10n)}$.
In particular, with probability at least $1-\nnew{\tau/10}$, all samples $X_1, \ldots, X_{n}$ from $f_0$ are in $S$.
\end{itemize}
\end{corollary}
\begin{proof}
From Lemma \ref{lem:DKS1}, we have that
$
\vol(S) = \vol(L_{f_0}(M_{f_0} e^{-z})) \leq O(z^d / M_{f_0}).
$
Also from Lemma \ref{lem:DKS1}, we have that
$
\Pr_{X \sim f_0}[f_0(X) \leq M_{f_0} / (100 n^4 /\tau^2)]  \leq \tau/(10n),
$
if we assume a sufficiently large constant is selected in the definition of $N_1$.
Taking a union bound over all samples, we get that with probability at least $1-\tau/10$, 
all of the $n$ samples are in $S$, as required.
\end{proof}


We can now state our main lemma establishing an upper bound
on the error of approximating the probability of every convex set:

\begin{lemma}\label{lem:alt_convex_set_prob}
For $n \geq N_1$, we have that 
with probability at least $1 - \tau/3$ over the choice of $X_1,\ldots,X_n$ drawn from $f_0$, 
for any convex set $C \subseteq \R^d$ it holds that
$\left|f_0(C) - f_n(C) \right| \leq \delta.$
\end{lemma}

The proof of Lemma~\ref{lem:alt_convex_set_prob} is deferred to Section~\ref{sec:full-proof}. 
In Section~\ref{sec:warmup}, we establish a weaker version of this lemma 
that requires more samples but has a simpler proof.
Combining Lemma \ref{lem:alt_convex_set_prob} with the observation 
that for any log-concave density $f$ and $t > 0$ we have that $L_f(t)$ is convex, 
we obtain the following corollary:

\begin{corollary}\label{cor:imp_convex_set_prob}
Let $n \geq N_1$.
\new{Conditioning on the event of Lemma~\ref{lem:alt_convex_set_prob},}
we have that for any $f \in \mathcal{F}_d$ and for any $t \geq 0$ 
it holds
$
\left|{\Pr}_{X\sim f_0}[f(X) \geq t] - {\Pr}_{X\sim f_n}[f(X) \geq t]\right| < \delta.$
\end{corollary}

\new{
We will require the following technical claim, which follows from standard properties
of Lebesgue integration (see Appendix~\ref{sec:main-appendix}):}
\begin{lemma}\label{lem:aux_exp}
Let $g,h:\mathbb{R}^d\to \mathbb{R}$ be \nnew{probability distributions}, and $\phi:\mathbb{R}\to \mathbb{R}$.
If ${\E}_{Y\sim g}[\phi(Y)]$, ${\E}_{Y\sim h}[\phi(Y)]$ are both finite, then
$
|{\E}_{Y\sim g}[\phi(Y)] - {\E}_{Y\sim h}[\phi(Y)]| \leq \int^{\infty}_{-\infty} | {\Pr}_{Y\sim g}[\phi(Y) < x] - {\Pr}_{Y\sim h}[\phi(Y) < x] | dx
$.
\end{lemma}

Our next lemma establishes a useful upper 
bound on the empirical error of the truncated likelihood of any log-concave density:
\begin{lemma}\label{lem:technical_f}
Let $n \geq N_1$ and $f \in \mathcal{F}_d$ with maximum value $M_f$. 
For all $\rho \in (0,M_f]$, \new{conditioning on the event of Corollary~\ref{cor:imp_convex_set_prob},} 
we have \[
|{\E}_{X\sim f_0}[\ln(\max(f(X), \rho))] - {\E}_{X\sim f_n}[\ln(\max(f(X), \rho))]| \leq \new{\delta} \cdot \ln (M_{f}/\rho) \;.
\]
\end{lemma}
\begin{proof}
Letting $h=f_0$, $g=f_n$, and $\phi(x)=\ln(\max(f(x),\rho))$, by Lemma \ref{lem:aux_exp} we have
\begin{align*}
|{\E}_{X\sim f_0}[&\ln(\max(f(X), \rho))] - {\E}_{X\sim f_n}[\ln(\max(f(X), \rho))]| \\
&\leq \int^{\infty}_{-\infty} \left|{\Pr}_{X\sim f_0}[\ln(\max(f(X), \ln \rho)) < t] - {\Pr}_{X\sim f_n}[\ln(\max(f(X), \rho)) < t] \right| dt \\ 
&= \int^{\ln M_f}_{-\infty} \left| {\Pr}_{X\sim f_0}[\max(\ln f(X), \ln \rho)) < t] - {\Pr}_{X\sim f_n}[\max(\ln f(X), {\ln} \rho)) < t] \right| dt \\ 
&= \int^{\ln M_f}_{\ln \rho} \left|{\Pr}_{X\sim f_0}[{\ln( f(X) )} < t] - {\Pr}_{X\sim f_n}[{\ln( f(X)  )} < t] \right| dt \\
&= \int^{\ln M_f}_{\ln \rho} \left|{\Pr}_{X\sim f_0}[f(X) < e^t] - {\Pr}_{X\sim f_n}[f(X) < e^t] \right| dt \\
&= \int^{\ln M_f}_{\ln \rho} \left|{\Pr}_{X\sim f_0}[f(X) \geq e^t] - {\Pr}_{X\sim f_n}[f(X) \geq e^t] \right| dt.
\end{align*}
Since we {conditioned on the event of Corollary~\ref{cor:imp_convex_set_prob}}, 
we have $\left|{\Pr}_{X\sim f_0}[f(X) \geq t] - {\Pr}_{X\sim f_n}[f(X) \geq t]\right| \leq \delta$ for all $t\geq 0$.
Therefore, we have that
$$\left|{\E}_{X\sim f_0}[\ln(\max(f(X), { \rho } ))] - {\E}_{X\sim f_n}[\ln(\max(f(X), \rho))]\right| 
\leq \int^{\ln M_f}_{\ln \rho} {\delta} dt  = \new{\delta} \cdot (\ln M_f - \ln\rho) \;,$$
which concludes the proof.
\end{proof}

For $f_0$ itself, we can use Hoeffding's inequality to get a bound on the empirical error of its likelihood:

\begin{lemma}\label{lem:f_0_bound}
Let $n \geq N_1$.
\new{Conditioning on the event of Corollary~\ref{lem:S_def}, with probability at least $1-\tau/3$ over $X_1, \ldots, X_n$, 
we have that}
\[
\left|\frac{1}{n} { \sum_{i=1}^n } \ln f_0(X_i) - {\E}_{X\sim f_0}\left[\ln f_0(X)\right]\right| \leq \eps/8 \;.
\]
\end{lemma}
We defer this proof to Appendix~\ref{sec:main-appendix}.
The following simple lemma shows that the MLE is supported
in the convex hull of the samples:

\begin{lemma}\label{lem:mle_support}
Let $n \geq 1$.
Let $X_1,\ldots,X_n$ be samples drawn from $f_0$, and $C$ be the convex hull of these samples.
Then, for all $x\in \mathbb{R}^d \setminus C$, we have $\mle(x)=0$.
\end{lemma}
We defer this proof to Appendix~\ref{sec:main-appendix}.
We need to truncate the likelihood at a density small enough to be ignored for $f$ close to $f_0$.
This motivates the following definition:

\begin{definition}\label{def:f_prime}
We define \nnew{$\tilde{f} : \mathbb{R}^d \rightarrow \mathbb{R}$ such that $\tilde{f}(x) \eqdef \max\{\pmin, \mle(x)\}$}.
\end{definition}

We show that this truncation and renormalization does not affect the MLE $\mle$ by much:

\begin{lemma}\label{lem:tv}
Let $n \geq N_1$.
Let $g(x) \eqdef \alpha \nnew{\tilde{f}(x) \mathbf{1}_S(x)}$, $\alpha \in [0,\infty)$, be such that $\int_{S} g(x) dx = 1$.
Conditioning on the event of Corollary~\ref{lem:S_def}, we have the following:
\begin{itemize}
\item[(a)] $1-\eps/32 \leq \alpha \leq 1$, and 
\item[(b)] $\TV(g, \mle) \leq {3\eps/64}$.
\end{itemize}
\end{lemma}
\begin{proof}
We start by proving (a).
By the definition of $g$ and Lemma \ref{lem:mle_support}, we have $\alpha = \alpha \int_{S} \mle(x) dx \leq \alpha \int_S f'(x) dx = \int_S g(x) dx = 1$, i.e., 
$\alpha \leq 1$.
Furthermore, \new{by the definition of $\pmin$ and Corollary~\ref{lem:S_def},} we have
\begin{align}
\pmin \cdot \vol(S) 
\leq \frac{M_{f_0}}{(100 n^4/\tau^2)} \cdot \frac{O((\ln(100 n^4/\tau^2))^d)}{M_{f_0}}
\leq {\eps/32}, \label{eq:p_min_vol}
\end{align}
and therefore
\begin{align*}
1 &= \int_S {g(x)} dx 
\leq \alpha \left( \int_S \pmin dx + \int_S \mle(x) dx \right) 
\leq \alpha ( \pmin \cdot \vol(S) + 1) 
\leq  {\alpha(\eps/32 + 1)}.
\end{align*}
From this it follows that
 {$\alpha \geq 1/(1+\eps/32) \geq 1 - \eps/32$}.
We have
\begin{align}
\TV(g, \mle) &= \frac{1}{2} \int_{\mathbb{R}^d} |g(x) - \mle(x)| dx = \frac{1}{2} \int_{S} |g(x) - \mle(x)| dx \;, \label{eq:tv1}
\end{align}
since \nnew{$g(x) = 0$ for $x \notin S$} and $\mle$ is supported in $S$ by Lemma \ref{lem:mle_support}.
We can then write
\begin{align*}
\frac{1}{2} \int_{S} |g(x) - \mle(x)| dx 
&= \frac{1}{2} \int_{S} |\alpha f'(x) - \mle(x)| dx \\
&\leq \frac{1}{2} \int_{S} |\alpha - 1| \cdot \mle(x) dx + \pmin \cdot \vol(S) \\
&\leq \frac{|\alpha - 1|}{2} \int_{S} \mle(x) dx +  {\eps/32} & \text{(from \eqref{eq:p_min_vol})} \\
&\leq \frac{|1 - \alpha|}{2} +  {\eps/32} \leq  {3\eps/64} \;,
\end{align*}
which completes the proof.
\end{proof}

To deal with the dependence on the maximum value of $f$ in Lemma~\ref{lem:technical_f}, 
we need to bound  {the maximum value of the MLE}.

\begin{lemma}\label{lem:mle_mf}
Let $n \geq N_1$.
Let $X_1, \ldots, X_n$ be samples drawn from $f_0$.
Then \new{conditioning on the events of Corollary~\ref{cor:imp_convex_set_prob} and Lemma~\ref{lem:f_0_bound}}, 
for any $f \in \mathcal{F}_d$ with maximum value $M_f$ such that 
$\ln (M_f/\pmin) \geq  {4\ln(100 n^4/\tau^2)}$, we have
$
\frac{1}{n} \sum_{i=1}^n \ln f(X_i) < \frac{1}{n} \sum_{i=1}^n \ln f_0(X_i).
$
\end{lemma}
This holds because a density $f$ with a large $M_f$ is small outside on a set of small volume, which most of the samples will be outside. We defer this proof to Appendix \ref{sec:main-appendix}.

We have now reached the final result of this section, from which Theorem \ref{thm:h-main-result} directly follows.
Combining previous  {lemmas}, we show that the likelihood under $f_0$ of the truncated 
MLE is close to that of $f_0$ and so they are close in KL divergence, 
which leads to a bound in the Hellinger distance of the MLE itself:

\begin{lemma}\label{lem:mle_dist_N_1}
Let $n \geq N_1$.
Let $X_1, \ldots, X_n$ be samples drawn from $f_0$.
With probability at least $1 - \tau$, we have that $h^2(f_0, \mle) \leq \eps$.
\end{lemma}
\begin{proof}
In this lemma, we will apply Lemmas \ref{lem:technical_f}, \ref{lem:f_0_bound}, \ref{lem:tv}, and \ref{lem:mle_mf}. 
By examining the conditions of these lemmas, it is easy to see that with probability at least $1 - \tau$ they all hold. 
We henceforth condition on this event.

Let $X_1, \ldots, X_n$ be samples drawn from $f_0$, let $\mle$ be as in Definition \ref{def:mle}.
Let $g$ and $\nnew{\tilde{f}}$ be as defined in Lemma \ref{lem:tv} and Definition \ref{def:f_prime}.
Let $S$ be as defined in Corollary~\ref{lem:S_def}
Then we have that
\begin{align*}
{\E}_{X \sim f_0}[\ln g(X)] &= {\E}_{X \sim f_0}[\ln(\alpha \nnew{\tilde{f}}(X))] \\
&\geq {\E}_{X \sim f_0}[\ln \nnew{\tilde{f}}(X)]  {- \eps/16} & \text{(since  {$\alpha > 1-\eps/32$})} \\
&\nnew{=} {\E}_{X \sim f_0}[\ln(\max\{\mle(X), \pmin\})]  {- \eps/16} \\
&\geq {\E}_{X \sim f_n}[\ln(\max\{\mle(X), \pmin\})]  {- 3\eps/16} & \text{(by Lemmas \ref{lem:technical_f} and \ref{lem:mle_mf})} \\
&\geq \frac{1}{n} \sum_i \ln \mle(X_i)  {- 3\eps/16} \\
&\geq \frac{1}{n} \sum_i \ln f_0(X_i)  {- 3\eps/16} \\ 
&\geq {\E}_{X \sim f_0}[\ln f_0(X)]  {- 5\eps/16}. & \text{(using Lemma \ref{lem:f_0_bound})}
\end{align*}
Thus, we obtain that
\begin{align}
\KL(f_0 || g) = {\E}_{X \sim f_0}[\ln f_0(X)] - {\E}_{X \sim f_0}[\ln g(X)] \leq  {5\eps/16}. \label{eq:kl_bound}
\end{align}
For the next derivation, we use that the Hellinger distance is related to the total variation distance and the Kullback-Leibler divergence in the following way: For probability functions $k_1, k_2 : \mathbb{R}^d \rightarrow \mathbb{R}$, we have that $h^2(k_1,k_2) \leq \TV(k_1,k_2)$ and $h^2(k_1,k_2) \leq \KL(k_1 || k_2)$.
Therefore, we have that
\begin{align*}
h(f_0, \mle) &\leq h(f_0, g) + h(g, \mle) \\
&\leq \KL(f_0 || g)^{1/2} + \TV(g, \mle)^{1/2} \\
&=  {(5\eps/16)^{1/2}} +  {(3\eps/64)^{1/2}} & \text{(by \eqref{eq:kl_bound} and Lemma \ref{lem:tv})} \\
&\leq \eps^{1/2} \;,
\end{align*}
concluding the proof.
\end{proof}

\section{Warmup for the Proof of Lemma \ref{lem:alt_convex_set_prob}}\label{sec:warmup}

For the sake of exposition of the main ideas used in the proof of Lemma \ref{lem:alt_convex_set_prob}, 
we first prove Lemma \ref{lem:convex_set_prob}, which achieves a weaker bound on the sample complexity, but has a significantly simpler proof.
Let us first give a brief, and somewhat imprecise, overview of the proof of Lemma \ref{lem:convex_set_prob}.
The high-level goal is to approximate some convex set $C\subseteq \mathbb{R}^d$ by some set, belonging to a family of low VC dimension.
We then can obtain the desired bound using Theorem \ref{thm:vc}.
To that end, we compute inner and outer approximations, $\cin$ and $\cout$, of $C$ via polyhedral sets with a small number of facets.
By Lemma \ref{lem:VC_polytopes}, we can argue that the VC dimension of this family is low.
We therefore obtain that $f_0$ and $f_n$ are close on the inner and outer approximations of $C$.
It remains to argue that the total difference between $f_0$ and $f_n$ in $\cout\setminus \cin$ is also small.
It thus suffices to bound the volume of $\cout\setminus \cin$.
This can be achieved by first defining some set $S\subseteq \mathbb{R}^d$ that excludes the tail of $f_0$.
Since $f_0$ is logconcave, we can show that $S$ has small volume.
The final bound is obtained by restricting the above argument on $C\cap S$.

Throughout this section, we define $N_2  {\eqdef} \Theta\left( 2^{O(d)}(d^{(2d+3)}/\eps) (\ln(d^{(d+1)}/(\eps\tau)) )^{(d+1)} \right)^{(d+5)/2}$.

We will require the following simple fact:

\begin{lemma}[see~\cite{ASE:92}]\label{lem:VC_polytopes}
Let $h,d\in  {\Z_+}$, and let ${\cal A}$ be the set of all convex 
polytopes in $\mathbb{R}^d$ with at most $h$ facets.
Then, the VC dimension of ${\cal A}$ is at most $2(d+1) h \log((d+1)h)$.
\end{lemma}


The main result of this section is the following:

\begin{lemma}\label{lem:convex_set_prob}
Let $n \geq N_2$.
With probability at least $1 - \frac{3\tau}{10}$ over the choice of  {$X_1,\ldots,X_n$}, 
for any convex set $C \subseteq \R^d$ it holds that 
$\left| f_0(C) - f_n(C) \right| < \delta$.
\end{lemma}
\begin{proof}
 {Recall that $z = \ln(100 n^4/\tau^2)$ and $S = L_{f_0}(M_{f_0} e^{-z})$}.
\nnew{Let $\mathcal{C}$ be the family of convex sets on $\mathbb{R}^d$.
For any $C \in \mathcal{C}$,} let $C' = C \cap S$.
Since $f_0$ is log-concave, it follows that $S$ is convex, and thus $C'$ is also convex.

Let ${\cal E}_1$ be the event that all samples  {$X_1,\ldots,X_n$} lie in $S$.
\nnew{Let $\mathcal{X} = X_1,\ldots,X_n$.}
By Corollary~\ref{lem:S_def}, we have
\begin{align}
\nnew{ {\Pr}_{\mathcal{X}\sim f_0}[{\cal E}_1] 
\geq 1-\tau/10. } \label{eq:pr_E1}
\end{align}
Conditioned on ${\cal E}_1$ occurring, \nnew{we have with probability 1, for any $C \in \mathcal{C}$, $f_n(C)=f_n(C')$.
In other words,
\begin{align}
{\Pr}_{\mathcal{X}\sim f_0}[\forall C \in \mathcal{C}, f_n(C \setminus C')=0 | {\cal E}_1] = 1. \label{eq:f_n_C}
\end{align}
}
From Corollary~\ref{lem:S_def}, we have ${\Pr}_{X\sim f_0}[f_0(X) \leq M_{f_0} / (100 n^4/\tau^2)] \leq \tau/(10n)$, and therefore
\begin{align}
f_0(C\setminus C') \leq f_0(\mathbb{R}^d\setminus S) \leq \tau/(10n) \leq \delta/5. \label{eq:f_0_C}
\end{align}
Combining \eqref{eq:pr_E1}, \eqref{eq:f_n_C}, \eqref{eq:f_0_C}, and letting 
$Q = \sup_{C \in \mathcal{C}} |f_0(C\setminus C') - f_n(C\setminus C')|$, we have that
\nnew{
\begin{align}
{\Pr}_{\mathcal{X}\sim f_0}\left[ Q \leq \delta/5\right] 
&\geq {\Pr}_{\mathcal{X}\sim f_0}\left[ Q \leq \delta/5 | {\cal E}_1 \right] \cdot {\Pr}_{\mathcal{X}\sim f_0}[{\cal E}_1] \notag \\
&\geq {\Pr}_{\mathcal{X}\sim f_0}\left[ \forall C \in \mathcal{C}, f_n(C\setminus C') = 0 | {\cal E}_1 \right] \cdot {\Pr}_{\mathcal{X}\sim f_0}[{\cal E}_1] \notag \\
 &\geq 1-\tau/10. \label{eq:prCCprime}
\end{align}
}
Let ${\cal A}$ be the set of convex polytopes in $\mathbb{R}^d$ 
with at most $H=(10 \kappa dz^{d}/\delta)^{(d-1)/2}$ facets, 
where $\kappa$ is the universal constant in Theorem \ref{thm:inside_poly}.
By Theorem \ref{thm:inside_poly}, there exist convex polytopes $T, T' \in {\cal A}$, 
with $T \subseteq C' \subseteq T'$, such that
$
\vol(C' \setminus T) 
\leq \frac{\delta}{10 z^d} \vol(S) 
\leq \frac{\delta}{10 M_{f_0}}
$
and
$
\vol(T'\setminus C') 
\leq \frac{\delta}{10 z^d} \vol(S) \leq 
\frac{\delta}{10 M_{f_0}}.
$
Therefore, since $M_{f_0}$ is the maximum value of $f_0$, we have
\begin{align}
f_0(C'\setminus T) &\leq \vol(C'\setminus T) \cdot M_{f_0} \leq \delta/10,  \label{eq:f0_Cp_T}
\end{align}
and
\begin{align}
f_0(T'\setminus C') &\leq \vol(T'\setminus C') \cdot M_{f_0} \leq \delta/10. \label{eq:f0_Tp_Cp}
\end{align}
 {Noting that $\E[|f_0(T)-f_n(T)|] \leq \E[||f_0 - f_n||_{\mathcal{A}}]$, by Theorem \ref{thm:vc} we have for some universal constant $\alpha$ that}
$
\E[|f_0(T)-f_n(T)|]
\leq  { \sqrt{ \alpha V / n } }
$.
The following claim is obtained via a simple calculation (see Appendix~\ref{sec:main-appendix}):

\begin{claim}\label{claim:easy_vc}
 {For $n \geq N_2$, we have that $\sqrt{\alpha V / n} \leq \delta / 10$}. 
\end{claim}

\nnew{
Let ${\cal E}_2$ be the event that $||f_0-f_n||_{\mathcal{A}} \leq 3\delta/10$.
By Claim~\ref{claim:easy_vc} and Theorem~\ref{thm:A-expect-bound} we have
\begin{align}
{\Pr}_{\mathcal{X}\sim f_0}[{\cal E}_2] 
&= 1 - {\Pr}_{\mathcal{X}\sim f_0}[||f_0-f_n||_{\mathcal{A}} > 3\delta/10] \notag \\
&\geq 1 - {\Pr}_{\mathcal{X}\sim f_0}[||f_0-f_n||_{\mathcal{A}} - \E[||f_0-f_n||_{\mathcal{A}}] > \delta/5] \notag \\
&\geq 1 - e^{-2n(\delta/5)^2} \notag \\
&\geq 1 - \tau/5. \label{eq:pr_E2}
\end{align}
For any choice of samples $X_1,\ldots,X_n$, we have
\begin{align}
f_n(C') &\geq f_n(T) & \text{(since $T\subseteq C'$)} \notag \\
 &\geq f_0(C') - f_0(C'\setminus T) - |f_0(T)-f_n(T)| \notag \\
 &\geq f_0(C') - \frac{\delta}{10} - |f_0(T)-f_n(T)|. & \text{(by \eqref{eq:f0_Cp_T})} \label{eq:fn_Cp_T_1}
\end{align}
In a similar way, using that $C'\subseteq T'$, we have
\begin{align}
f_n(C') 
 &\leq f_0(C') + \frac{\delta}{10} + |f_0(T')-f_n(T')|. & \text{(by \eqref{eq:f0_Tp_Cp})} \label{eq:fn_Cp_Tp_1}
\end{align}
By \eqref{eq:fn_Cp_T_1} and \eqref{eq:fn_Cp_Tp_1} and the union bound, we obtain
\begin{align}
|f_n(C')-f_0(C')| \leq \frac{\delta}{10} + \max\left\{ |f_0(T)-f_n(T)|, |f_0(T')-f_n(T')| \right\}. \label{eq:Pr_fn_Cp_f0_Cp}
\end{align}
Combining \eqref{eq:prCCprime}, \eqref{eq:pr_E2}, \eqref{eq:Pr_fn_Cp_f0_Cp}, 
and letting $Q' = \sup_{C \in \mathcal{C}} |f_n(C)-f_0(C)|$, we get
\begin{align}
{\Pr}_{\mathcal{X}\sim f_0}[ Q' \leq 2\delta/5] \notag 
&\geq {\Pr}_{\mathcal{X}\sim f_0}[( \sup_{C \in \mathcal{C}} |f_n(C\setminus C')-f_0(C\setminus C')| \leq \delta/5) \wedge Q \leq 3\delta/10)] \notag\\
&\geq {\Pr}_{\mathcal{X}\sim f_0}[( \sup_{C \in \mathcal{C}} |f_n(C\setminus C')-f_0(C\setminus C')| \leq \delta/5) \wedge ( ||f_n - f_0||_{\mathcal{A}} \leq 3\delta/10)] \notag\\
 &\geq 1-3\tau/10 \notag,
\end{align}
which concludes the proof.
}
\end{proof}

\section{Conclusions} \label{sec:conc}

In this paper, we gave the first sample complexity upper bound for the MLE of
multivariate log-concave densities on $\R^d$, for any $d \geq 4$. 
Our upper bound agrees with the previously known lower bound up to a multiplicative factor of
$\tilde O_d (\eps^{-1})$.

A number of open problems remain:
What is the {\em optimal} sample complexity of the multivariate log-concave MLE?
In particular, is the log-concave MLE sample-optimal for $d \geq 4$?
Does the multivariate log-concave MLE have similar adaptivity properties
as in one dimension? And is there a polynomial time algorithm to compute it?



\bibliography{allrefs}

\appendix

\section{Proof of Lemma \ref{lem:alt_convex_set_prob}}\label{sec:full-proof}

We are now ready to prove the main technical part of our work, which is Lemma \ref{lem:alt_convex_set_prob}.
The proof builds upon the argument used in the proof of Lemma \ref{lem:convex_set_prob}, which achieves a weaker sample complexity bound.
Recall that in the proof of Lemma \ref{lem:convex_set_prob} we use inner and outer polyhedral approximations of $C$, restricted on some appropriate bounded $S\subseteq \mathbb{R}^d$.
The main difference in the proof of Lemma \ref{lem:alt_convex_set_prob} is that we now use roughly $O(\log n)$ inner and outer polyhedral approximations of intersections of $C$ with different super-levelsets of $f_0$.
We need slightly more samples due to the higher number of facets, and consequently higher VC dimension of the resulting approximations.
However, since we use a finer discretization of the values of $f_0$, we incur lower error in total.

The following Lemma is implicit in \cite{DiakonikolasKS17-lc}. We reproduce its proof for completeness in Appendix~\ref{sec:main-appendix}.

\begin{lemma}\label{lem:vc_sets}
Let $L,H \in  {\Z_+}$.
We define the set $\mathcal{A}_{H,L}$, elements of which are defined by the following process: Starting with $L$ convex polytopes each with at most $H$ facets, all combinations of intersection, difference, and union of these polytopes are elements of $\mathcal{A}_{H,L}$.
If $V$ is the VC dimension of $\mathcal{A}_{H,L}$, then $V/\log(V) = O(dLH)$.
\end{lemma}

We are now prepared to present the proof of Lemma \ref{lem:alt_convex_set_prob}.
Let 
\[
S_i = L_{f_0}(M_{f_0} e^{-i})
\]
and let $S_0 = \emptyset$.
\nnew{
Let $L = \ln(100 n^4/\tau)$. Note that by Lemma \ref{lem:DKS1}, we have that ${\Pr}_{X\sim f_0}[f_0(X) \leq M_{f_0} e^{-z}] = O(d)^d e^{-z/2}$ and thus
\begin{align*}
{\Pr}_{X\sim f_0}[X \notin S_L] &= {\Pr}_{X\sim f_0}[f_0(X) < M_{f_0} e^{-L}] \leq \frac{\tau}{10n}.
\end{align*}
Let ${\cal E}_1$ be the event that all samples  {$X_1, \ldots, X_n$} lie in $S_L$. 
Let $\mathcal{X} = X_1, \ldots, X_n$.
We have that
\begin{align}
 {{\Pr}_{\mathcal{X}\sim f_0}[{\cal E}_1] \geq 1-\tau/10.} \label{eq:pr_E4}
\end{align}
} 
\nnew{
Let $\mathcal{C}$ be the set of convex sets in $\mathbb{R}^d$.
For any $C \in \mathcal{C}$, for all $i \in [L]$, let 
\[
C_i = C \cap S_i.
\]
Note that, conditioned on ${\cal E}_1$ occurring, we have with probability 1 that, for all $C \in \mathcal{C}$, $f_n(C) = f_n(C_L)$. In other words, 
\begin{align}
{\Pr}_{\mathcal{X}\sim f_0}[\forall C \in \mathcal{C}, f_n(C \setminus C_L) = 0 |  {{\cal E}_1}] = 1. \label{eq:C_CL_FN}
\end{align}
Furthermore, by our choice of $L$ we have  {$f_0(\R^d \setminus S_L) \leq \frac{\tau}{10n}$}, and therefore
\begin{align}
f_0(C \setminus C_L) \leq \frac{\tau}{10n} \leq \delta / 5. \label{eq:C_CL_f0}
\end{align}
Combining \ref{eq:pr_E4}, \ref{eq:C_CL_FN}, \ref{eq:C_CL_f0}, and letting 
$Q = \sup_{C \in \mathcal{C}} |f_0(C \setminus C_L) - f_n(C \setminus C_L)|$, we have
\begin{align}
{\Pr}_{\mathcal{X}\sim f_0} \left[ Q \leq \delta / 5 \right]
&\geq {\Pr}_{\mathcal{X}\sim f_0} \left[ Q \leq \delta / 5 \big| {{\cal E}_1} \right] \cdot {\Pr}_{\mathcal{X}\sim f_0}[ {{\cal E}_1}] \notag \\
&\geq {\Pr}_{\mathcal{X}\sim f_0}[\forall C \in \mathcal{C}, f_n(C \setminus C_L) = 0 | {\cal E}_1] \cdot {\Pr}_{\mathcal{X}\sim f_0}[ {{\cal E}_1}] \notag \\
&\geq 1 - \tau/10. \label{eq:C_L_bound}
\end{align}
}

Using Theorem \ref{thm:inside_poly}, for $i \in [L]$ let $\pini, \pouti$ be convex polytopes with $H = (10 \kappa d / \delta)^{(d-1)/2}$ facets, where $\kappa$ is the universal constant from Theorem \ref{thm:outside_poly}, such that $\pini \subseteq C_i \subseteq \pouti$,
\begin{align}
\vol(C_i \setminus \pini)  \leq \delta \cdot \vol(C_i) / 10 \leq \delta \cdot \vol(S_i) / 10, \label{eq:ci_pini_bound}
\end{align}
and 
\begin{align}
\vol(\pouti \setminus C_i)  \leq \delta \cdot \vol(C_i) / 10 \leq \delta \cdot \vol(S_i) / 10. \label{eq:pouti_ci_bound}
\end{align}

Let 
\[
\cin = \bigcup_{i \in [L]} \pini.
\]

 {For any $i\in [L]$, let $P^S_i$} be a convex polytope with at most $H$ facets such that $P^S_i \subseteq S_i$ and $\vol(S_i \setminus P^S_i) \leq \delta \cdot \vol(S_i)/10$.

Let 
\[
S'_i = \bigcup_{1 \leq j \leq i} P^S_j
\]
and $S'_0 = \emptyset$.
Let 
\[
\cout = \bigcup_{i \in [L]} (\pouti \setminus S'_{i-1}).
\]

We will now show that $\cin$ and $\cout$ satisfy the following conditions:
\begin{enumerate}
\item $\cin \subseteq C_L \subseteq \cout$.
\item $f_0(\cout \setminus C_L) < \delta/2$.
\item $f_0(C_L \setminus \cin) < \delta/2$.
\end{enumerate}

\begin{figure}[t]
\begin{center}
\scalebox{0.3}{\includegraphics{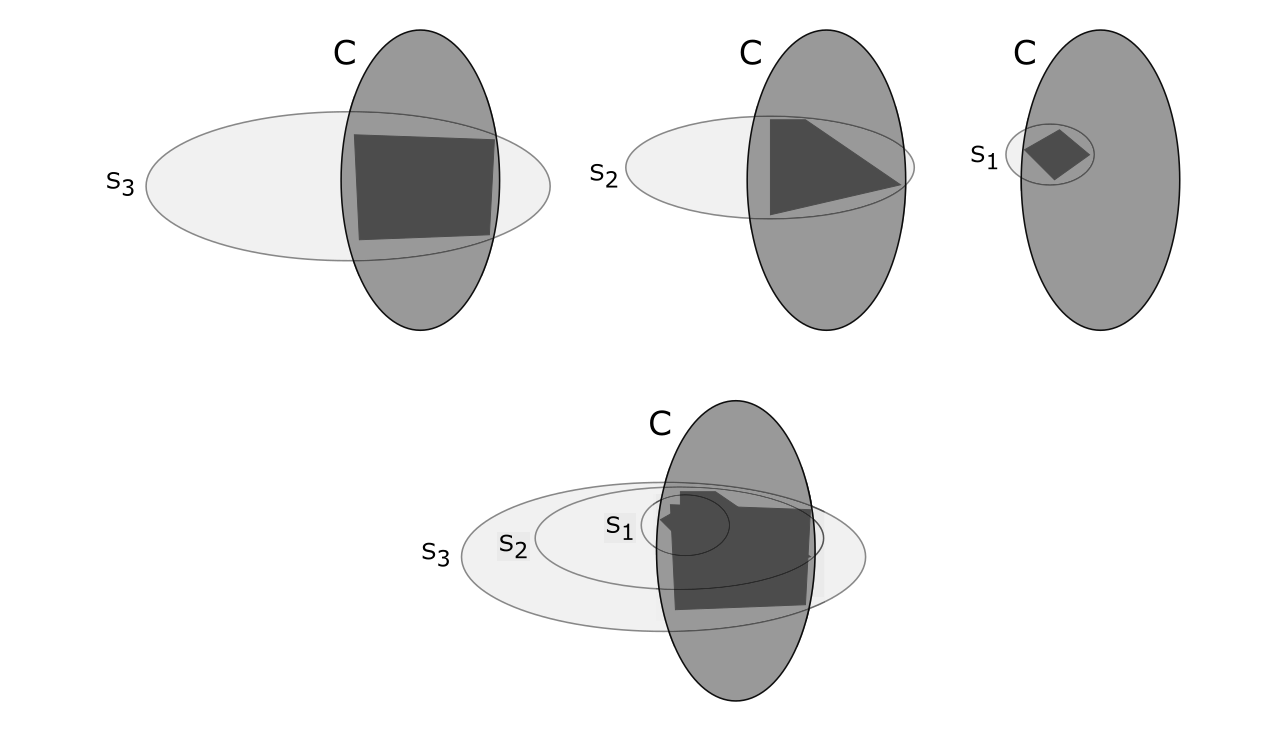}}
\caption{Constructing $\cin$. For each set $S_i$, a convex polytope approximating $C \cap S_i$ from the inside is found, and $\cin$ is formed by taking the union of these convex polytopes. }
\label{fig:c_in}
\end{center}
\end{figure}

First, we consider $\cin$.
Since $\pini \subseteq C_i \subseteq C_L$ for all $i \in [L]$, it follows that $\bigcup_{i \in [L]} \pini = \cin \subseteq C_L$.
Observe that by the above definitions, we have that
\begin{align}
(C_L \setminus \cin) \cap (S_i \setminus S_{i-1})
&\subseteq (C_L \setminus \cin) \setminus S_{i-1} \subseteq (C_L \setminus \pini) \setminus S_{i-1} \label{eq:cl_cin_s_i-1}.
\end{align}
From \eqref{eq:cl_cin_s_i-1}, we therefore have
\begin{align}
(C_L \setminus \cin)
&= \bigcup_{i \in [L]} \left[ (C_L \setminus \cin) \cap (S_i \setminus S_{i-1}) \right] \subseteq \bigcup_{i \in [L]} (C_i \setminus \pini) \setminus S_{i-1}, \label{eq:ci_pini}
\end{align}
and so
\begin{align}
f_0(C_L \setminus \cin)
&\leq \sum_{i\in [L]} f_0((C_i \setminus \pini) \setminus S_{i-1}) & \text{(by \eqref{eq:ci_pini})} \notag \\
&\leq \sum_{i \in [L]} \vol((C_i \setminus \pini) \setminus S_{i-1}) M_{f_0} e^{-(i-1)} \notag \\
&\leq \sum_{i \in [L]} \vol(C_i \setminus \pini) M_{f_0} e^{-(i-1)} \notag \\
&\leq \sum_{i \in [L]} (\delta/10) \vol(S_i) M_{f_0} e^{-(i-1)} & \text{(by \eqref{eq:ci_pini_bound})} \notag \\
&\leq (\delta/10) \sum_{i \in [L]} \vol(L_{f_0}(M_{f_0} e^{-i})) M_{f_0} e^{-(i-1)} \notag \\
&\leq (\delta /10) \int_0^{M_{f_0}} \vol(L_{f_0}(y)) dy < \delta/2. \label{eq:imp_f0_Cp_T}
\end{align}

\begin{figure}
\begin{center}
\scalebox{0.3}{\includegraphics{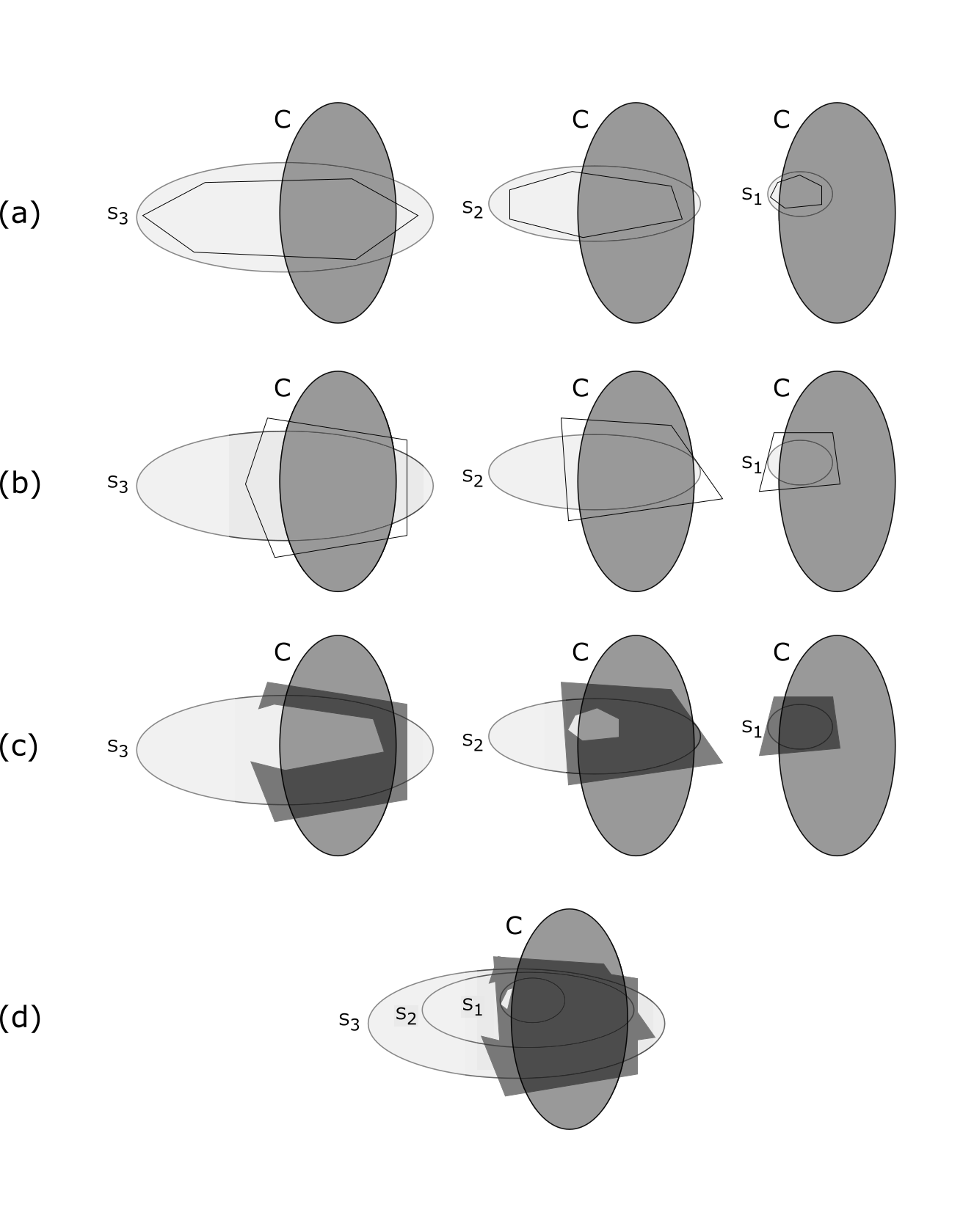}}
\caption{Constructing $\cout$. For each set $S_i$, a convex polytope approximating $S_i$ from the inside is found ($P^S_i$, see row \textbf{(a)}), and a convex polytope approximating $C \cap S_i$ from the outside is found ($\pouti$, see row \textbf{(b)}). For each $i$, the set $\pouti \setminus (\cup_{j = 1}^{i-1} P^S_j)$ is constructed (see row \textbf{(c)}), and the union of these sets finish the construction of $
\cout$ (see row \textbf{(d)}).}
\label{fig:c_out}
\end{center}
\end{figure}

Now we consider $\cout$.
Let $x \in C_L$. Then there exists $i \in [L]$ such that $x \in S_i$ and $x \notin S_{i-1}$.
Thus $x \in \pouti$ and $x \notin S'_{i-1}$, from which we have that $x \in \cout = \bigcup_{i \in [L]} (\pouti \setminus S'_{i-1})$. Therefore $C_L \subseteq \cout$.
Let $y \in \cout \setminus C_L$. 
From the definition of $\cout$, there must exist some $i \in [L]$ such that $y \in \pouti \setminus S'_{i-1}$. If $y \in \pouti \setminus C_i$, we are done. Suppose that $y \notin \pouti \setminus C_i$. Since we have that $y \in \pouti$, we must also have that $y \in C_i$. But $C_i \subseteq C_L$, and we began with $y \in \cout \setminus C_L$, which makes a contradiction. Therefore,
\begin{align}
\cout \setminus C_L \subseteq \cup_{i \in [L]} \left(\pouti \setminus C_i\right). \label{eq:pouti_ci}
\end{align}
Thus, we have that
\begin{align}
f_0(\cout \setminus C_L)
&\leq \sum_{i \in [L]} f_0(\pouti \setminus C_i) & \text{(by \eqref{eq:pouti_ci})} \notag \\
&\leq \sum_{i \in [L]} \vol(\pouti \setminus C_i) M_{f_0} e^{-(i-1)} \notag \\
&\leq \sum_{i \in [L]} (\delta/10) \vol(S_i) M_{f_0} e^{-(i-1)} &\text{(by \eqref{eq:pouti_ci_bound})} \notag \\
&\leq (\delta/10) \sum_{i \in [L]} \vol(L_{f_0}(M_{f_0} e^{-i})) M_{f_0} e^{-(i-1)} \notag \\
&\leq (\delta/10) \int_0^{M_{f_0}} \vol(L_{f_0}(y)) dy < \delta / 2. \label{eq:imp_f0_Tp_Cp}
\end{align}

We define the set $\mathcal{A}$, elements of which are defined by the following process: Starting with $2L$ convex polytopes each with at most $H$ facets, all combinations of intersection, difference, and union of these convex polytopes are elements of $\mathcal{A}$.
Then for any convex set $C$ with $\cin, \cout$ as defined above, we have that $\cout, \cin \in \mathcal{A}$.
From Lemma \ref{lem:vc_sets}, we have that if $V$ is the VC dimension of $\mathcal{A}$, then
\[
V/\ln(V) = O(d L H).
\]
Using Theorem \ref{thm:vc}, we have  {for some universal constant $\alpha$} that 
\begin{align}
\E[|f_0(\cin) - f_n(\cin)|] &\leq \E[||f_0 - f_n||_\mathcal{A}] = \sqrt{\frac{ {\alpha V}}{n}}. \label{eq:hard_vc}
\end{align}

The following claim is obtained via a simple calculation (see Appendix~\ref{sec:main-appendix}):

\begin{claim}\label{claim:hard_vc}
 {For $n \geq N_1$ we have that $\sqrt{\frac{ {\alpha V}}{n}} \leq \delta / 10$.}
\end{claim}

\nnew{
Let ${\cal E}_2$ be the event that $|| f_0 - f_n ||_{\mathcal{A}} \leq \delta / 2$. 
 {Then by \eqref{eq:hard_vc}, Claim~\ref{claim:hard_vc}, and Theorem \ref{thm:A-expect-bound}, we have that}
\begin{align}
{\Pr}_{\mathcal{X}\sim f_0}[{\cal E}_2] &= 1 - {\Pr}_{\mathcal{X}\sim f_0}[|| f_0 - f_n ||_{\mathcal{A}} > \delta / 2] \notag \\
&\geq 1 - {\Pr}_{\mathcal{X}\sim f_0}[|| f_0 - f_n ||_{\mathcal{A}} - \E[|| f_0 - f_n ||_{\mathcal{A}}] > \delta / 10] \notag \\
&\geq 1 - e^{-2n(\delta/10)^2} 
\notag \\
&\geq 1 - \tau/10. \label{eq:imp_cin_bound}
\end{align}
}

\nnew{This next claim follows from \eqref{eq:imp_f0_Cp_T} and \eqref{eq:imp_f0_Tp_Cp}}.
The full proof can be found in Appendix~\ref{sec:main-appendix}.
\nnew{
\begin{claim}\label{claim:hard_vc_cl}
 {If ${\cal E}_1$ and ${\cal E}_2$ hold, we have that $\sup_{C \in \mathcal{C}}|f_n(C_L)-f_0(C_L)| \leq 7\delta/10$}.
\end{claim}
}

\nnew{
Combining \eqref{eq:C_L_bound}, \eqref{eq:imp_cin_bound}, {Claim \ref{claim:hard_vc_cl}}, and letting 
$Q' = \sup_{C \in \mathcal{C}} |f_n(C)-f_0(C)|$, we get
\begin{align}
{\Pr}_{\mathcal{X}\sim f_0}[ Q' \leq \delta] 
&\geq {\Pr}_{\mathcal{X}\sim f_0}[ \sup_{C \in \mathcal{C}} |f_n(C\setminus C_L)-f_0(C\setminus C_L)| \leq \delta/5) \wedge  \sup_{C \in \mathcal{C}} |f_n(C_L)-f_0(C_L)| \leq 7\delta/10)] \notag\\
&\geq 1-\frac{\tau}{10}- \frac{\tau}{5} \notag\\
&\geq 1-3\tau/10 \notag,
\end{align}
}
which concludes the proof.

\section{Deferred Proofs}\label{sec:main-appendix}


\subsection{Proof of Lemma \ref{lem:DKS1}}

W.l.o.g.~we may assume that $f(0)=M_f$.
We let $R = L_{f}(M_{f} / e)$. Then using the fact that if $y \leq M_{f} / e$ then $R \subseteq L_{f}(y)$, we have that
\begin{align}
1 &= \int_{\mathbb{R}_+} \vol(L_{f}(y))dy 
\geq \int_{0 \leq y \leq M_{f} / e} \vol(L_{f}(y))dy 
\geq \int_{0 \leq y \leq M_{f} / e} \vol(R)dy 
= \frac{M_{f}}{e} \cdot \vol(R) \label{eq:volR}
\end{align}
Suppose that $f(x)\geq M_f e^{-w}$, for some $x\in \mathbb{R}^d$.
By the definition of log-concavity we have $f(x/w) \geq f(0)^{(w-1)/w}f(x)^{1/w}$.
By the assumption we get
$f(x/w) \geq M_f^{(w-1)/w} (M_f/e^w)^{1/w} = M_f^{(w-1)/w} M_f^{1/w} / e = M_f / e$.
Thus $x/w \in R$, and so $x \in wR$. 
Therefore $L_f(M_f e^{-w}) \subseteq wR$.
Thus by \eqref{eq:volR} we get
\begin{align}
\vol(L_{f}(M_{f} e^{-w})) \leq \vol(wR) \leq \nnew{w^d \cdot \vol(R)} = w^d / M_{f}, \label{eq:volLf}
\end{align}
which proved the first part of the assertion.

It remains to prove the second part.
We have
\begin{align*}
{\Pr}_{X\sim f}[f(X) \leq M_{f} e^{-z}] &\leq \int_{0}^{M_{f} e^{-z}} \vol(L_{f}(y)) dy &  \\
&= \int_{z}^{\infty} \vol(L_{f}(M_{f}e^{-x})) M_{f} e^{-x} dx  & \text{(setting $y=M_f e^{-x}$)} \\
&\leq \int_{z}^{\infty} O(x^d / M_{f}) M_{f} e^{-x} dx & \text{(by \eqref{eq:volLf})} \\
&= \int_{z}^{\infty} O(x^d e^{-x}) dx \\
&\leq \int_{z}^{\infty} O(d)^d e^{-x/2} dx & \text{(since $e^{x/2} \geq (x/2)^d / d!$)} \\
&= O(d)^d e^{-z/2},
\end{align*}
which concludes the proof.


\subsection{Proof of Lemma \ref{lem:aux_exp}}
We begin with a few common definitions and observations.
If $X$ is a random variable defined on a probability space $(\Omega, \Sigma, P)$, then the expected value $\E[X]$ of $X$ is defined as the Lebesgue integral
\[
\E[X] = \int_{\Omega} X(\omega)dP(\omega).
\]
Next, we define two functions 
\[
X_+(\omega) = \max(X(\omega),0)
\]
and
\[
X_-(\omega) = -\min(X(\omega),0).
\]
We observe that these functions are both measurable (and therefore also random variables), and that $\E[X] = \E[X_+] - \E[X_-]$.
Finally, we observe that if $X : \Omega \rightarrow \R_{\geq 0} \cup \{\infty\}$ is a non-negative random variable then
\[
\E[X] = \int_{0}^{\infty} \Pr[X > x] dx.
\]
Similarly, if $X : \Omega \rightarrow \R_{\geq 0} \cup \{-\infty\}$ is a non-positive random variable then
\[
\E[X] = -\int_{-\infty}^{0} \Pr[X < x] dx.
\]

Applying the definitions and observations of the previous paragraph, we have the following derivation:
\begin{equation*}
\begin{split}
{\E}_{Y\sim g}[\phi(Y)] - & {\E}_{Y\sim h}[\phi(Y)] \\
&= \left( {\E}_{Y\sim g}[\phi(Y)_+] - {\E}_{Y\sim g}[\phi(Y)_-] \right) - \left( {\E}_{Y\sim h}[\phi(Y)_+] - {\E}_{Y\sim h}[\phi(Y)_-] \right) \\
&= \left( {\E}_{Y\sim g}[\phi(Y)_+] + {\E}_{Y\sim g}[-\phi(Y)_-] \right) - \left( {\E}_{Y\sim h}[\phi(Y)_+] + {\E}_{Y\sim h}[-\phi(Y)_-] \right) \\
&= \left( \int_{0}^{\infty} {\Pr}_{Y\sim g}[\phi(Y)_+ > x]dx + \int_{-\infty}^{0} {\Pr}_{Y\sim g}[-\phi(Y)_- < x]dx \right) \\ &\quad - \left( \int_{0}^{\infty} {\Pr}_{Y\sim h}[\phi(Y)_+ > x]dx + \int_{-\infty}^{0} {\Pr}_{Y\sim h}[-\phi(Y)_- < x]dx \right) \\
&= \left( \int_{0}^{\infty} {\Pr}_{Y\sim g}[\phi(Y) > x]dx + \int_{-\infty}^{0} {\Pr}_{Y\sim g}[\phi(Y) < x]dx \right) \\ &\quad - \left( \int_{0}^{\infty} {\Pr}_{Y\sim h}[\phi(Y) > x]dx + \int_{-\infty}^{0} {\Pr}_{Y\sim h}[\phi(Y) < x]dx \right) \\
&= \int_{0}^{\infty} {\Pr}_{Y\sim g}[\phi(Y) > x] - {\Pr}_{Y\sim h}[\phi(Y) > x] dx \\ &\quad + \int_{-\infty}^{0} {\Pr}_{Y\sim g}[\phi(Y) < x] - {\Pr}_{Y\sim h}[\phi(Y) < x] dx \\
&= \int_{0}^{\infty} (1 - {\Pr}_{Y\sim g}[\phi(Y) < x]) - (1 - {\Pr}_{Y\sim h}[\phi(Y) < x]) dx \\ &\quad + \int_{-\infty}^{0} {\Pr}_{Y\sim g}[\phi(Y) < x] - {\Pr}_{Y\sim h}[\phi(Y) < x] dx \\
&= \int_{0}^{\infty} {\Pr}_{Y\sim h}[\phi(Y) < x] - {\Pr}_{Y\sim g}[\phi(Y) < x]) dx \\ &\quad + \int_{-\infty}^{0} {\Pr}_{Y\sim g}[\phi(Y) < x] - {\Pr}_{Y\sim h}[\phi(Y) < x] dx \\
&\leq \int_{0}^{\infty} \left| {\Pr}_{Y\sim g}[\phi(Y) < x] - {\Pr}_{Y\sim h}[\phi(Y) < x] \right| dx \\ &\quad + \int_{-\infty}^{0} \left| {\Pr}_{Y\sim g}[\phi(Y) < x] - {\Pr}_{Y\sim h}[\phi(Y) < x] \right| dx \\
&= \int_{-\infty}^{\infty} \left| {\Pr}_{Y\sim g}[\phi(Y) < x] - {\Pr}_{Y\sim h}[\phi(Y) < x] \right| dx.
\end{split}
\end{equation*}
A symmetric argument shows that
\[
{\E}_{Y\sim h}[\phi(Y)] - {\E}_{Y\sim g}[\phi(Y)] \leq \int_{-\infty}^{\infty} \left| {\Pr}_{Y\sim h}[\phi(Y) < x] - {\Pr}_{Y\sim g}[\phi(Y) < x] \right| dx,
\]
concluding the proof.

\subsection{ {Proof of Lemma \ref{lem:f_0_bound}}}
\nnew{Recall that $z = \ln (100 n^4 /\tau^2)$, $S = L_{f_0}(M_{f_0} e^{-z})$, and $\pmin = M_{f_0} /(100n^4/\tau^2)$.}
Note that for any $x \in S$, we have $f_0(x) \geq \new{\pmin}$ by construction.
\new{Since we have conditioned on the event of Corollary~\ref{lem:S_def} holding, it follows that for each $i \in [n]$, 
$f_0(X_i) \geq \new{\pmin}$.}
Therefore, letting $\rho \eqdef \new{\pmin}$, we have
\begin{align}
\left| \frac{1}{n} {\sum_{i=1}^n \ln( f_0(X_i) ) } - {\E}_{X\sim f_0}\left[\ln f_0(X)\right] \right| \notag 
&= \left|\frac{1}{n} { \sum_{i=1}^n } \ln(\max(f_0(X_i), \rho)) - {\E}_{X\sim f_0}\left[\ln f_0(X)\right]\right| \notag \\
&\leq \left|\frac{1}{n} { \sum_{i=1}^n } \ln(\max(f_0(X_i), \rho)) - {\E}_{X\sim f_0}\left[\ln(\max(f_0(X), \rho)) \right]\right| \notag \notag \\
&~~ + \left| {\E}_{X\sim f_0}\left[\ln(\max(f_0(X), \rho)) \right] - {\E}_{X\sim f_0}\left[\ln f_0(X) \right]\right| \notag \\
&\leq \left|\frac{1}{n} { \sum_{i=1}^n } \ln(\max(f_0(X_i), \rho)) - {\E}_{X\sim f_0}\left[\ln(\max(f_0(X), \rho)) \right]\right| \notag \\
&~~ + \int_{-\infty}^{\ln\rho} \Pr[\ln f_0(X) { \leq } T] dT \;. \label{eq:hoeff_exp}
\end{align}
By Hoeffding's inequality we have 
\begin{align}
\Pr&\left[\left|\frac{1}{n} { \sum_{i=1}^n } \ln(\max(f_0(X_i), \rho)) - {\E}_{X\sim f_0}\left[\ln(\max(f_0(X), \rho)) \right]\right| > \frac{\eps}{16} \right] \notag \\
&\leq 2 \exp \left(\frac{-2 n^2 (\eps / 16)^2}{n \cdot (\ln M_{f_0} - \ln \rho)^2}\right) \notag \\
&\leq 2 \exp \left(\frac{-n \eps^2/16^2}{(\ln(100 n^4/\tau^2))^2}\right) \notag \\
&\leq \tau/3  \label{eq:exp_pmin}  \;. & \text{(since $n \geq N_1$)}
\end{align}
Next we have
\begin{align}
\int_{-\infty}^{\ln \rho} {\Pr}_{{X\sim f_0}}[\ln f_0(X) { \leq } T] dT &\leq \int_{0}^{\infty} {\Pr}_{{X\sim f_0}}[\ln f_0(X) { \leq } \ln\rho - y] dy \notag &\text{(setting $y = \ln\rho - T$)} \\
&\leq \int_{0}^{\infty} O(d)^d (\rho / M_{f_0})^{1/2} e^{-y/2} dy  & \text{(by Lemma \ref{lem:DKS1})} \notag \\
&= \int_{0}^{\infty} O(d)^d \frac{\tau}{10n^2} e^{-y/2} dy & \text{($\rho = M_{f_0} / (100 n^4/\tau^2)$)} \notag \\
&\leq 2 \cdot O(d)^d \frac{\tau}{10n^2} \notag \\
&\leq \eps / 16. & \text{(since $n\geq N_1$)} \label{eq:exp_tail}
\end{align}
By applying \eqref{eq:exp_pmin} and \eqref{eq:exp_tail} to bound \eqref{eq:hoeff_exp} from above, 
with probability at least $1-\tau/3$ we have  
that $$\left|\frac{1}{n} { \sum_{i=1}^n } \ln f_0(X_i) - {\E}_{X\sim f_0}\left[\ln f_0(X)\right]\right| \leq \eps/8 \;,$$
which concludes the proof.

\subsection{ {Proof of Lemma \ref{lem:mle_support}}}
Suppose there exists $x \in \mathbb{R}^d \setminus C$ such that $\mle(x) > 0$.
Then, we have that $L_{\mle}(\mle(x)) \setminus C \neq \emptyset$ and thus 
$\int_{\mathbb{R}^d \setminus C} \mle(x) dx > 0$.
From this, it follows that
$\int_{{C}} \mle(x) dx < 1$,
and so there exists some $\alpha > 1$ 
such that $\alpha \int_{{C}} \mle(x) dx = 1$.
{Let $\hat{g}_n : C \rightarrow \R$ be such that $\hat{g}_n = \alpha \cdot \mle |_C$}.
Since $C$ is a convex set and $\int_{C} \hat{g}_n(x) dx = 1$, we have that $\hat{g}_n$ is a log-concave density.
Observe that
\begin{align}
\frac{1}{n} { \sum_{i=1}^n } \log(\hat{g}_n(X_i)) &= \frac{1}{n} { \sum_{i=1}^n } \log(\alpha \hat{f}_n (X_i)) 
> \frac{1}{n} { \sum_{i=1}^n } \log(\mle (X_i)) \;, \label{mle_S}
\end{align}
where we used that $\alpha>1$.
By definition, $\mle$ maximizes $\frac{1}{n} { \sum_{i=1}^n } \log(f(X_i))$ 
over all log-concave densities $f$, which contradicts \eqref{mle_S}.
Therefore, for all $x \in \mathbb{R}^d \setminus C$, we have that $\mle(x) = 0$.

\subsection{ {Proof of Lemma \ref{lem:mle_mf}}}
\nnew{This lemma holds because for a density $f$ with a large maximum value $M_f$, $f$ is small outside a set of small volume, and most of the samples drawn from $f_0$ will be outside this set.}
Let
\[
\gamma = \exp{\left( 2 \left( \frac{1}{n}  {\sum_{i=1}^{n} \ln f_0(X_i)} - \frac{1}{2}\ln M_f - 1 \right) \right)}
\]
and
\[
A = L_f(\gamma).
\]
If we have that $\vol(A) \cdot M_{f_0} \leq 1/3$, then it follows that $f_0(A) \leq 1/3$. Since $f$ is log-concave, $A$ is a convex set, and since we condition on Corollary~\ref{cor:imp_convex_set_prob} holding, we have with probability $1$ that $|f_0(A) - f_n(A)| < \delta < 1/6$. Therefore, we have that $f_n(A) < 1/2$, in which case at least $1/2$ of the samples  {$X_1, \ldots, X_n$} are not contained within $A$. Thus, we have that
\begin{align*}
\frac{1}{n}  {\sum_{i=1}^{n}} \ln f(x) &\leq \frac{1}{2}\ln \gamma + \frac{1}{2}\ln M_f \\
&= \frac{1}{2} \cdot 2 \left( \frac{1}{n}  {\sum_{i=1}^{n} \ln f_0(X_i)} - \frac{1}{2}\ln M_f - 1 \right) + \frac{1}{2}\ln M_f \\
&< \frac{1}{n}  {\sum_{i=1}^{n} \ln f_0(X_i)}.
\end{align*}
Now we check to see how large $M_f$ must be to ensure that $\vol(A) \cdot M_{f_0} \leq 1/3$.
We have that
\begin{align*}
\vol(A) \cdot M_{f_0} &= \vol(L_f(\gamma)) \cdot M_{f_0} \\
&= \vol\left(L_f\left( M_f \cdot \exp{\left( \frac{2}{n}  {\sum_{i=1}^{n}  \ln f_0(X_i)} - 2  -2 \ln M_f \right) } \right)\right) \cdot M_{f_0} \\
&\leq \frac{M_{f_0}}{M_f} \cdot O\left( \left(2 - \frac{2}{n}  {\sum_{i=1}^{n}  \ln f_0(X_i)}  +2 \ln M_f \right)^d \right)\;. &\text{(by Lemma \ref{lem:DKS1})}
\end{align*}
Since we condition on the event of Lemma~\ref{lem:f_0_bound} holding, we have with probability $1$ that
\[
\frac{1}{n}  {\sum_{i=1}^{n}  \ln f_0(X_i)} \geq {\E}_{X \sim f_0} \left[\ln f_0(X)\right] - \eps \geq \ln \pmin - \eps,
\]
and so we have that
\begin{align*}
\vol(A) \cdot M_{f_0} &\leq \frac{M_{f_0}}{M_f} \cdot O\left( \left(2 + 2\ln M_f - 2\ln \pmin + 2\eps \right)^d \right) \\
&< \frac{M_{f_0}}{M_f} \cdot O\left( \left(2\ln M_f - 2\ln M_{f_0} + 3\ln(n^4 100/\tau^2) \right)^d \right). \\
\end{align*}
The following claim follows by a simple calculation:
\begin{claim}\label{claim:m_f_2}
 {If $\ln(M_f / M_{f_0}) \geq 3\ln(100n^4/\tau^2)$, then $\vol(A) \cdot M_{f_0} \leq 1/3$.}
\end{claim}
\begin{proof}
Recall that
\begin{align*}
\vol(A) \cdot M_{f_0} &\leq \frac{M_{f_0}}{M_f} \cdot O\left( \left(2 + 2\ln M_f - 2\ln \pmin + 2\eps \right)^d \right) \\
&< \frac{M_{f_0}}{M_f} \cdot O\left( \left(2\ln M_f - 2\ln M_{f_0} + 3\ln(n^4 100/\tau^2) \right)^d \right). \\
\end{align*}
We search for $M_f$ such that $\vol(A) \cdot M_{f_0} \leq 1/3$. It is sufficient for $M_f$ to satisfy, for some constant $c>1$,
\nnew{\begin{align}
M_{f_0} / M_f \cdot c \left(2\ln M_f - 2\ln M_{f_0} + 3\ln(n^4 100/\tau^2) \right)^d &\leq 1/3 \notag \\
\ln\left( (M_{f_0} / M_f) \cdot c \left(2\ln(M_f / M_{f_0}) + 3\ln(n^4 100/\tau^2) \right)^d \right) &\leq \ln(1/3) \notag \\
\ln\left( M_{f_0} / M_f \right) + \ln c + \ln\left( \left(2\ln(M_f / M_{f_0}) + 3\ln(n^4 100/\tau^2) \right)^d \right) &\leq \ln(1/3) \notag \\
\ln\left( \left(2\ln(M_f / M_{f_0}) + 3\ln(n^4 100/\tau^2) \right)^d \right) + \ln(3c) &\leq \ln\left( M_f / M_{f_0} \right) \notag \\
d \ln\left( 2\ln(M_f / M_{f_0}) + 3\ln(n^4 100/\tau^2) \right) + \ln(3c) &\leq \ln\left( M_f / M_{f_0} \right). \label{eq:m_f_1}
\end{align}}
If we have $M_f$ such that $\ln(M_f / M_{f_0}) \geq 3\ln(n^4 100/\tau^2)$, 
and a sufficiently large constant is chosen for  {$N_1$} so that $\ln(3c) \leq \ln(n^4 100/\tau^2)$, 
then \eqref{eq:m_f_1} becomes
\begin{align}
\nnew{d \ln\left( 3\ln(M_f / M_{f_0}) \right)} &\leq 2\ln\left( M_f / M_{f_0} \right). \label{eq:m_f_2}
\end{align}
The next inequality is equivalent to \eqref{eq:m_f_2} {:}
\begin{align*}
\nnew{\left( 3\ln(M_f / M_{f_0}) \right)^{d/2}} &\leq M_f / M_{f_0} 
\end{align*}
We note that the derivative of \nnew{$(3 \ln x)^{d/2}$} is 
\[
\nnew{\frac{3^{d/2} d (\ln x)^{d/2 - 1}}{2x}}.
\]
We also note that for \nnew{$x = (3d)^{d/2 + 1}(\ln(9d))^{d/2+1}$} we have that
\[
\nnew{\frac{3^{d/2} d (\ln x)^{d/2 - 1}}{2x} \leq 1}.
\]
and
\nnew{\begin{align*}
(3\ln x)^{d/2} &= [3(d/2 + 1)\ln(9d\ln(9d)]^{d/2} \\
&= 3^{d/2} \cdot d^{d/2} \cdot [2\ln(9d)]^{d/2} \\
&\leq x.
\end{align*}}
Therefore, assuming sufficiently large constants are chosen in the definition of  {$N_1$}, if 
\[
\ln(M_f / M_{f_0}) \geq 3\ln(n^4 100/\tau^2)
\]
then $\vol(A) \cdot M_{f_0} \leq 1/3$.
\end{proof}
 {Therefore, for $\ln(M_f / M_{f_0}) \geq 3\ln(100n^4/\tau^2)$ we have that}
$
\frac{1}{n}  {\sum_{i=1}^{n} \ln f(X_i) } < \frac{1}{n} \sum_i \ln f_0(x_i)
$
 {and}
\begin{align*}
\ln M_f - \ln \pmin 
&= \ln(M_f / M_{f_0}) +  {\ln(100n^4/\tau^2)}
\geq  {4}\ln(100n^4/\tau^2) \;,
\end{align*}
concluding the proof.

\subsection{ {Proof of Claim~\ref{claim:easy_vc}}}
By Lemma \ref{lem:VC_polytopes} we have that the VC dimension of ${\cal A}$ is $V\leq 2(d+1)H \ln((d+1)H)$, and so
$V 
\leq (10 \kappa)^{(d+1)/2} d^{(d+5)/2} (\ln(100 n^4/\tau^2))^{d}/\delta)^{(d+1)/2}$.
Noting that $\E[|f_0(T)-f_n(T)|] \leq \E[||f_0 - f_n||_{\mathcal{A}}]$, by Theorem \ref{thm:vc} we get that
\begin{align}
\E[|f_0(T)-f_n(T)|]
&\leq \sqrt{\frac{O(V)}{n}} \notag \\
&\leq \sqrt{ \frac{ O\left( (10 \kappa)^{(d+1)/2} d^{(d+5)/2} (\ln(100 n^4/\tau^2))^{d}/\delta)^{(d+1)/2} \right) }{n} }. \notag
\end{align}
For the next part we want that $\E[|f_0(T)-f_n(T)|] \leq \delta / 10$.
This holds when 
\begin{align*}
n &= \Omega\left( (d/\eps) (\ln(100 n^4/\tau^2))^{(d+1)} \right)^{(d+5)/2} 
\end{align*}
If $n \geq b \left(c^{d+1} (d^{(2d+3)}/\eps) (\ln(d^{(d+1)}/(\eps\tau)))^{(d+1)}  \right)^{(d+5)/2}$ for some constants $b > 1, c \geq 100 \ln c$, then we have
\begin{align*}
(d/\eps) (\ln(100 n^4/\tau^2))^{(d+1)}
&\leq (d^{(2d+3)}/\eps) (100 \ln c)^{(d+1)} \left( \ln(d^{(d+1)}/(\eps\tau)) \right)^{(d+1)} \\
&\leq c^{d+1} (d^{(2d+3)}/\eps) (\ln(d^{(d+1)}/(\eps\tau)) )^{(d+1)}
\end{align*}
and therefore $n = \Omega\left( (d/\eps) (\ln(100 n^4/\tau^2)^{(d+1)} \right)^{(d+5)/2}$ as desired.
Therefore, for $n \geq N_2$ we have
\begin{align}
\E[|f_0(T)-f_n(T)|] &\leq \delta / 10. \label{eq:VC_T}
\end{align}

\subsection{Proof of Lemma~\ref{lem:vc_sets}}
Consider an arbitrary set $T$ of $t$ points in $\mathbb{R}^d$. We wish to bound the number of possible distinct sets that can be obtained by the intersection of $T$ with a set in $\mathcal{A}_{H,L}$.
We note that $\mathcal{A}_{H,L}$ can also be constructed in the following manner: Take an arrangement consisting of at most  {$H \cdot L$} hyperplanes. This arrangement partitions $\mathcal{R}^d$ into a set of components. Then, the union of subsets of these components are elements of $\mathcal{A}_{H,L}$.
Any halfspace can be perturbed, without changing its intersection with $T$, so that its boundary intersects $d'+1$ points in $T$, where $d'\leq d$ is the dimension of the affine subspace spanned by $T$.
Any such subset uniquely determines the intersection of the halfspace with $T$.
Therefore, the number of possible intersections with a set of size $t$ is at most $O(t)^d$.
It follows then that the number of possible intersections of any $A \in \mathcal{A}_{H,L}$ and any set of size $t$ is at most
$(O(t)^d)^{LH} \leq O(t)^{dLH}$.
If $\mathcal{A}$ has VC dimension $t$, then is must be that
$O(t)^{dLH} \geq 2^t$,
and therefore
$t/\log(t) = O(dLH)$.

\subsection{ {Proof of Claim~\ref{claim:hard_vc}}}
Recalling that $L= \ln(100 n^4/\tau^2)$ and $H = (10 \kappa d / \delta)^{(d-1)/2}$, we have that
\begin{align}
V / \ln(V) &= O\left( d \cdot \ln(100 n^4/\tau^2) \cdot (10 \kappa d / \delta)^{(d-1)/2} \right) \notag \\
&= O\left( (10 \kappa)^{(d-1)/2} d^{(d+1)/2} \ln(100 n^4/\tau^2) / \delta^{(d-1)/2} \right). \label{eq:improved_v}
\end{align}
We note that
\begin{align*}
\ln\left( (10 \kappa)^{(d-1)/2} d^{(d+3)/2} (\ln(100 n^4/\tau^2))^2 / \delta^{(d-1)/2} \right)
&\leq \frac{d-1}{2} \ln\left( (10 \kappa) d^{3} (\ln(100 n^4/\tau^2))^{6} / \delta \right) \\
&\leq d \ln\left( (10 \kappa) d^{3} (\ln(100 n^4/\tau^2))^{6} / \delta \right) \\
&\leq c d \ln( \ln(100 n^4/\tau^2) )
\end{align*}
for some sufficiently large constant $c$. 
Therefore, letting
\[
V = O\left( (10 \kappa)^{(d-1)/2} d^{(d+3)/2} (\ln(100 n^4/\tau^2))^2 / \delta^{(d-1)/2} \right)
\]
satisfies \eqref{eq:improved_v}.
Therefore, we have that 
\begin{align}
\sqrt{\frac{\alpha V}{n}}
&= \sqrt{\frac{\alpha \cdot O\left( (10 \kappa)^{(d-1)/2} d^{(d+3)/2} (\ln(100 n^4/\tau^2))^2 / \delta^{(d-1)/2} \right) }{ n } }, \notag
\end{align}
and thus when
\begin{align}
n &= \Omega\left( (10 \kappa)^{(d-1)/2} d^{(d+3)/2} (\ln(100 n^4/\tau^2))^{(d+7)/2} / \eps^{(d+3)/2} \right) \label{eq:v_improved_n}
\end{align}
we have that $\sqrt{\frac{\alpha V}{n}} \leq \delta/10$.
To simplify \eqref{eq:v_improved_n}, we note that the 
\[
d^{(d+3)/2} (\ln(100 n^4/\tau^2))^{(d+7)/2} / \eps^{(d+3)/2} \leq \left( (d/\eps) (\ln(100 n^4/\tau^2))^{2} \right)^{(d+3)/2}.
\] 
Thus, if we let $n = (c (d^{2}/\eps) (\ln(d/(\eps\tau)))^{3})^{(d+3)/2}$ for some large constant $c$, then we have that
\begin{align*}
\ln(100 n^4/\tau^2) 
&= \frac{d+3}{2} \ln(100c (d^{2}/\eps) (\ln(d/(\eps\tau)))^{3}) + \ln(1/\tau^2) \\
&\leq c' d \ln(d/(\eps\tau))
\end{align*}
for some large constant $c'$.
Thus, assuming a sufficiently large constant is chosen, for $n \geq N_1$ we have that \eqref{eq:v_improved_n} holds, and therefore
\begin{align}
\sqrt{\frac{\alpha V}{n}} \leq \delta / 10. \label{eq:improved_expect}
\end{align}

\subsection{ {Proof of Claim~\ref{claim:hard_vc_cl}}}
\nnew{
For any choice of the samples  {$X_1,\ldots,X_n$}, we have
\begin{align}
f_n(C_L) &\geq f_n(\cin) & \text{(since $\cin \subseteq C_L$)} \notag \\
 &\geq f_0(\cin) - |f_0(\cin)-f_n(\cin)| \notag \\
 &= f_0(C_L) - f_0(C_L\setminus \cin) - |f_0(\cin)-f_n(\cin)| \notag \\
 &\geq f_0(C_L) - \frac{\delta}{2} - |f_0(\cin)-f_n(\cin)|. & \text{(by  \eqref{eq:imp_f0_Cp_T})} \label{eq:imp_fn_Cp_T_1}
\end{align}
Similarly, we have
\begin{align}
f_n(C_L) &\leq f_n(\cout) & \text{(since $C_L\subseteq \cout$)} \notag \\
 &\leq f_0(\cout) + |f_0(\cout)-f_n(\cout)| \notag \\
 &= f_0(C_L) + f_0(\cout\setminus C_L) - |f_0(\cout)-f_n(\cout)| \notag \\
 &\leq f_0(C_L) + \frac{\delta}{2} + |f_0(\cout)-f_n(\cout)|. & \text{(by  \eqref{eq:imp_f0_Tp_Cp})} \label{eq:imp_fn_Cp_Tp_1}
\end{align}
Combining \eqref{eq:imp_fn_Cp_T_1} and \eqref{eq:imp_fn_Cp_Tp_1}, we therefore have that
\begin{align*}
|f_n(C_L) - f_0(C_L)| \leq \frac{\delta}{2} + \max\left\{ |f_0(\cin)-f_n(\cin)|, |f_0(\cout)-f_n(\cout)| \right\} \label{eq:fn_f0_max}
\end{align*}
From this, we therefore have that
\begin{align*}
\sup_{C \in \mathcal{C} } |f_n(C_L)-f_0(C_L)| \leq 7\delta/10,
\end{align*}
concluding the proof.
}

\end{document}